\documentclass[11pt,reqno]{amsart}
\usepackage{amsmath,amsfonts,amscd,amsxtra,amssymb,latexsym,calc,xfrac}

\usepackage[a4paper,bookmarks,
bookmarksnumbered,%
 colorlinks=true,%
 linkcolor=blue,%
 citecolor=blue,%
 filecolor=blue,%
 urlcolor=blue,%
]
{hyperref}

 \usepackage[T1]{fontenc}
 \usepackage[latin1]{inputenc}
 \usepackage{times}
\usepackage{euscript} 
 
\DeclareMathOperator{\Sing}{Sing}

\def\ra{\rightarrow}
\def\cal{\mathcal} 
\def\wt{\widetilde}
\def\ol{\overline}

\def\CC{\mathbb{C}}
\def\PP{\mathbb{P}}
\def\QQ{\mathbb{Q}}
\def\ZZ{\mathbb{Z}}
  
\def\RR{\mathbb{R}}

\def\OO{\cal O} 
\def\XX{\cal X}

\def\yy{\cal Y}

\def\s-{\setminus}

\def\bcp{\mathbb C\mathbb P}

\newcommand{\bigslant}[2]{{\raisebox{.2em}{$#1$}\left/\raisebox{-.2em}{$#2$}\right.}}

\newtheorem{main}{Theorem}

\newtheorem{thm}{Theorem}[section]
\newtheorem{prop}[thm]{Proposition}

\newtheorem{defn}[thm]{Definition}

\newtheorem{cor}[thm]{Corollary}
\newtheorem{lem}[thm]{Lemma}

\newtheorem{rmk}[thm]{Remark}

\newtheorem{examples}[thm]{Examples}

\numberwithin{equation}{section}

\begin{document}

\title{ALE Ricci-flat K\"ahler surfaces and weighted projective spaces}

\author{R. R\u asdeaconu, I. {\c S}uvaina}

\address{
        Department of Mathematics,1326 Stevenson Center, Vanderbilt University, Nashville, TN, 37240}
        
\email{rares.rasdeaconu@vanderbilt.edu}

\email{ioana.suvaina@vanderbilt.edu}



\date{\today}

\begin{abstract}
We show that the explicit ALE Ricci-flat K\"ahler metrics constructed 
by Eguchi-Hanson, Gibbons-Hawking, Hitchin and Kronheimer, 
and their free quotients are metrics obtained by Tian-Yau techniques. 
The proof relies on a construction of  good compactifications  
of  $\QQ-$Gorenstein deformations of  quotient surface singularities 
as log del Pezzo surfaces with only cyclic quotient singularities at infinity.
\end{abstract}

\maketitle

\section{Introduction}
\label{intro}

In his 1978 ICM plenary lecture \cite{yau}, Yau raised several 
questions regarding the existence of compactifications of 
complete Ricci-flat K\"ahler manifolds. One of the questions asks 
about the existence of compactifications in the complex analytic sense, 
and another if the anticanonical line bundle of this compactification 
has ``good properties''. The identification of the required properties 
became clear later, when in their joint work \cite{ty1,ty2}, Tian and Yau 
provided sufficient conditions such that complete Ricci-flat K\"ahler 
metrics exist on the complement of a divisor. In this paper, we address 
these questions for surfaces equipped with asymptotically locally Euclidean 
(ALE) metrics.

The classification of ALE Ricci-flat K\"ahler surfaces was accomplished 
by Kronheimer \cite{kron1,kron2} in the simply-connected case, 
and completed  by the second author \cite{ALE}  in the non-simply connected 
case (see also \cite{wri}). More precisely, we have the following:

\begin{thm}[\cite{kron1, kron2, ALE}]
\label{classif}
Let  $(M, J, g,\omega_g)$  be a smooth  ALE Ricci-flat K\"ahler surface 
asymptotic to $\CC^2/G,$ where $G$ is a finite subgroup of $U(2)$ 
acting freely on ${\CC^2}\setminus\{0\}$. Then, the complex manifold 
$(M, J)$ can be obtained as the minimal resolution of a fiber of a 
one-parameter $\QQ-$Gorenstein deformation of the quotient singularity 
$\CC^2/G.$ Given the K\"ahler class $\Omega=[\omega_g]\in H^2(M,\RR),$ 
then $g$ is the unique ALE Ricci-flat K\"ahler metric  in the class. 

Moreover, any complex surface $(M, J)$ obtained by the above construction 
admits a unique ALE Ricci-flat K\"ahler metric in any K\"ahler class $\Omega.$ 
\end{thm}

The underlying complex surfaces in the above classification are in direct 
correspondence \cite{ALE} with the list of quotient singularities which 
admit $\QQ-$Gorenstein smoothings, which is due to Koll\'ar and 
Shepherd-Barron \cite{ksb}.  Following their terminology, we call such 
singularities of class $T$. The possible singularities are either rational double 
points, i.e. singularities of type $A_k, D_k, E_6, E_7$ and $E_8$, or finite 
cyclic singularities of the type $\dfrac1{dn^2}(1,dnm-1).$ The rational double 
points correspond to the case when the surfaces $M$ are simply connected, 
and then the metrics are hyperk\"ahler. They are  associated to Gorenstein 
smoothings and have trivial canonical line bundle. In the second case,  
the surfaces $M$ have finite cyclic fundamental group and the metrics 
are non-hyperk\"ahler. They are $\QQ-$Gorenstein smoothings, and have 
torsion canonical line bundle.

The ALE Ricci-flat K\"ahler metrics on open complex surfaces were explicitly  
constructed by Eguchi-Hanson, Gibbons-Hawking, Hitchin and Kronheimer 
\cite{egha,giha,hit,kron1} in the simply connected case. In \cite{ALE}, the 
second author completes the list in the non-simply connected case by adding 
the free quotients of certain $A-$type manifolds. The above classification is 
based on the theory of twistor spaces \cite{kron1,kron2, ALE}.

Another  method of constructing complete Ricci-flat K\"ahler metrics on 
non-compact manifolds is by solving  the complex Monge-Amp\`ere equations. 
The techniques are due to Tian-Yau \cite{ty1, ty2}, Bando-Kobayashi \cite{bako}, 
and Joyce \cite{joyce}. In particular, Tian, Yau and Bando, Kobayashi prove the 
existence of Ricci-flat K\"ahler metrics on the complement of a divisor in a 
complex orbifold under certain conditions. In general, these metrics are not ALE, 
and if one insists that the ambient space is a smooth surface, the examples are scarce 
as it can be seen from Lemma \ref{smoothcase} (see also \cite{ty1}). 
As we are only interested in manifolds with ALE Ricci-flat K\"ahler metrics, 
the relevant results are in the context of the complement of a divisor in an 
orbifold surface, and are due to Tian and Yau  \cite{ty2}. 
A more detailed description of their results is included in section \ref{T-Ysec}.
The metrics constructed by Tian and Yau are obtained by analytical methods, 
and the existence of the metric is given implicitly. We analyze the ALE K\"ahler 
Ricci flat surfaces from their perspective, and we answer Yau's questions \cite{yau}:

\begin{main}
\label{metrics}
Let $(M, J, g)$ be an ALE Ricci-flat K\"ahler surface. Then there exists a 
complex compactification $(\ol M, \bar J, D),$ where $\ol M$ is an orbifold 
surface and $D=\ol M\setminus M$ is the divisor at infinity, such that 
$D$ is an admissible, almost ample, admits a K\"ahler-Einstein metric 
and $-K_{\ol M}= \beta [D], \beta>1.$ In particular, any  ALE Ricci-flat 
K\"ahler metric $g$ can be obtained as a Tian-Yau metric.
\end{main}

As an immediate consequence of Theorem \ref{metrics}, 
we obtain the following result:

\begin{cor}
\label{equiv-met}
The Tian-Yau method rediscovers the metrics constructed by Eguchi-Hanson, 
Gibbons-Hawking, Hitchin, Kronheimer, and their finite free quotients.
\end{cor}

These results are an immediate continuation of the second author's work  
\cite{ALE} and fill in a gap in the understanding of the ALE Ricci-flat K\"ahler 
surfaces.

\bigskip

For ALE hyperk\"ahler surfaces, Kronheimer shows that the canonical model of 
the surface  is an affine hypersurface \cite{kron1}. This is essential in showing 
that complex analytic compactifications exist, and such examples appear in 
\cite{saito} (for some choice of a complex structure on $A,D,E$ cases) and  
\cite{craig} (for $A-$type surfaces). In this paper, we prove that there is a 
compactification for any ALE Ricci-flat K\"ahler surface. We consider compactifications 
of a  fiber of a $\QQ-$Gorenstein deformation as a hypersurface in a weighted 
projective space.  The properties of the  compactifications  are summarized in the 
following theorem, and for a more detailed description see Propositions 
\ref{compactific}, \ref{compact2}, and \ref{comp-def-non-cyclic}:

\begin{main}
\label{compactification}
Let $M$ be a fiber of a $\QQ-$Gorenstein deformation of a singularity of class $T.$ 
Then $M$ embeds into a log del Pezzo surface, $\ol M,$ as the complement of a 
smooth, rational curve, which is a rational multiple of the anticanonical divisor. 
The singularities along the divisor at infinity are all isolated finite cyclic quotients. 
Moreover, if $M$ is associated to a  finite cyclic quotient singularity, then there are 
infinitely many minimal  compactifications with the above properties. 
\end{main}

We say that a compactification is minimal if there is no rational component 
of the divisor at infinity of self-intersection $(-1)$ and passing only through 
smooth points of $\ol M.$  We recall  \cite{km} that a normal complex surface 
$\ol M$ with at worst log terminal singularities, i.e. quotient singularities, 
is called a log del Pezzo surface if its anti-canonical divisor $-K_{\ol M}$ is ample. 
We should point out that our constructions verify stronger conditions: if we denote 
by $D$ the curve at infinity, then the $\QQ-$Cartier divisors $D$ and 
$-(K_{\ol M}+D)$ are both ample.

Proposition \ref{compactific} describes an infinity of compactifications in the 
case of finite cyclic groups of $\dfrac1{dn^2}(1,dnm-1)-$type, in particular 
$A-$type when $n=1.$ Among these, for $n=1$ or $n=2$ there exists a unique 
one which can be used in conjunction with the Tian-Yau metric construction. 
For all other cases, the key condition that the divisor at infinity admits a 
K\"ahler-Einstein metric is not satisfied. We provide a second construction 
for compactifications of $\dfrac1{dn^2}(1,dnm-1)-$type surfaces in Section 
\ref{secondconstr}, as hypersurfaces in quotients of weighted projective spaces. 
Again we obtain infinitely many compactifications in the complex analytic sense, 
among which there exists only one which can be used to prove Theorem \ref{metrics}.

Given a singularity of type $T,$ and a $\QQ-$Gorenstein smoothing, 
the underlying smooth manifold of a generic fiber is the Milnor fiber of the singularity. 
An arbitrary fiber of a deformation might admit singularities, which are all 
rational double points. Hence, if we consider the associated minimal resolution 
we obtain a manifold diffeomorphic to the Milnor fiber. The general construction 
of considering a deformation followed by the minimal resolution of the 
rational double points, exhibits a family of complex structures on the Milnor fiber. 
In the special case, when the deformation is singular, the compactification 
$\ol M$ is not log del Pezzo, but nevertheless satisfies the required conditions 
of the Tian-Yau construction. Throughout this paper we emphasize which 
particular complex structure we consider.

In general, the compactifications of ALE K\"ahler Ricci-flat surfaces are rarely smooth. 
We discuss the cases when they are smooth in the last section of the article.

\subsection*{Notations and conventions}

\begin{enumerate}
\item By an $m-$ dimensional quotient singularity we mean a germ of an analytic 
space $(X,0),$ with $X = \CC^m/G,$ where the group $G\subseteq GL(m,\CC),$ 
and $0\in X$ is the representative of the $G-$orbit of $\{0\}\in \CC^m.$ If $G$ is 
the multiplicative group of the $n^{th}-$roots of unity, after a linear base change, 
we can assume  that the action is diagonal.

\item Let $n$ be a positive integer,  $\mu_n$ the multiplicative group of the 
$n^{th}-$roots of unity, and $\epsilon\in \mu_n$ a generator. We denote by the symbol 
$\displaystyle \frac1n(a_1,\dots,a_m)$ the action of the group $\mu_n$ on $\CC^m$
defined by 
$$
\epsilon(z_1,\dots,z_m)=(\epsilon^{a_1}z_1,\dots,\epsilon^{a_m}z_m),
$$ 
where $(a_1,\dots,a_m)\in \ZZ^m.$ 
We refer to the corresponding quotient space as a singularity of the type 
$\displaystyle \frac1n(a_1,\dots,a_m).$
Whenever necessary, we include the coordinates 
$$
\displaystyle \CC^m_{\bf z}/\frac1n(a_1,\dots,a_m),
$$ 
where  ${\bf z}=(z_1,\dots,z_m)\in \CC^m,$ to denote the cyclic quotient  singularity.

If $\rho$ is a different choice of the generator of the group $\mu_n,$ there exists an 
integer $k,~\gcd(k,n)=1,$ such that $\epsilon=\rho^k,$ and we obtain an equivalent 
notation of the singularity of the form $\displaystyle \frac1n(b_1,\dots,b_m),$ where 
$b_i=ka_i\mod n$ for $i=1,\dots, m.$

Let $f\in \CC[z_1,\dots,z_m].$ If $\displaystyle (f=0)\subseteq \CC^m$ is invariant under 
the above action of $\mu_n,$ we will denote by $\displaystyle (f=0)/\mu_n$ the induced 
quotient. If necessary, we explicitly include in the notation the action of $\mu_n,$ as above.

\item Let $(w_0,w_1,\dots,w_m)$ be an $(m+1)-$tuple of positive integers. 
The weighted projective space $\PP(w_0,\dots,w_m)$ is defined as the 
quotient of $\CC^{m+1}\setminus\{0\}$ by the $\CC^{*}-$action given by
$$
\lambda(z_0,z_1,\dots,z_m)=(\lambda^{w_0}z_0,\lambda^{w_1}z_1,\dots,\lambda^{w_m}z_m).
$$
Following \cite{fle}, we say that the weighted projective space 
$\PP(w_0,\dots,w_m)$ is {\em well-formed} if 
$\gcd(w_0,\dots,{\widehat w_i},\dots,w_m)=1,$ 
for each $i=0,\dots, m.$

The  weighted projective space is covered by the standard charts 
$
U_{z_i}=(z_i\neq 0)\simeq \CC_{\bf Z_{i}}^m/\frac1{w_i}
(w_0,\dots,{\widehat w_i},\dots, w_m)
$ 
centered at  
$P_i=[0:\ldots:1:\ldots:0],~i=0,\dots,m.$ The affine coordinates 
$
{\bf Z_i}=(Z_{0i},\dots,{\widehat  Z_{ii}},\dots, Z_{mi})
$ 
satisfy 
$$
Z_{ji}^{w_i}=\frac{z_j^{w_i}}{z_i^{w_j}}, ~j=0,\dots,{\widehat i},\dots, m,
$$
and are well-defined up to the corresponding action of $\mu_{w_i}.$ 
Whenever the coordinates are relevant in the descriptions of the spaces 
involved, we indicate them as 
$$
\PP_{[z_0:\dots:z_n]}(w_0,\dots,w_n).
$$
However, 
to simplify the notations we omit them when it is clear from the context.

\item All of the varieties discussed in this paper have only mild singularities. 
In particular, they are all $\QQ-$ factorial \cite{km}. We will not distinguish 
between their $\QQ-$Cartier divisors and Weil divisors with rational coefficients.

\end{enumerate}

\section{A compactification of an arbitrary fiber of a deformation}
\label{compact}

We begin by recalling the terminology and some general results contained 
in \cite{ksb,manetti}.  

\begin{defn}
A normal variety $X$ is $\QQ-$Gorenstein if it is Cohen-Macaulay and a 
multiple of the canonical divisor is Cartier.
\end{defn}
\begin{defn}
A flat map $\pi:\XX\ra \Delta\subseteq \CC$ is called a one-parameter 
$\QQ-$Gorenstein smoothing of a normal singularity $(X,x)$ if $\pi^{-1}(0)=X$ 
and there exists $U\subseteq \Delta$ an open neighborhood of $0$ such that 
the following conditions are satisfied. 
\begin{itemize}

\item [ i)] $\XX$ is $\QQ-$Gorenstein,

\item [ ii)] The induced map $\XX\ra U$ is surjective,

\item [ iii)] $X_t=\pi^{-1}(t)$ is smooth for every $t\in U\setminus\{0\}.$

\end{itemize}
\end{defn}

The following result of Koll{\'a}r and Shepherd-Barron \cite {ksb} gives a 
complete description of the singularities admitting a one-parameter 
$\QQ-$Gorenstein smoothing:

\begin{prop}[Koll{\'a}r,~Shepherd-Barron \cite{ksb}]
\label{descsing}
The quotient singularities admitting a one-parameter $\QQ-$Gorenstein 
smoothing are the following:
\begin{itemize}

\item [ 1)] Rational double points;

\item [ 2)] Cyclic singularities of the type 
$\displaystyle \frac{1}{dn^2}(1,dnm-1),$ 
for $d>0,~n\geq 1,$ and $(m,n)=1.$ 

\end{itemize}
\end{prop}

For convenience, we recall that the rational double points are 
isolated quotients of $\CC^2$ by finite subgroups of $SU(2).$ 
They are classified by their types $A,D$ or $E.$ The singularities 
of type $A_{k-1}$ are cyclic quotient singularities of the type 
$\displaystyle \frac1k(1,-1),$ while the other rational double 
points are quotients of $\CC^2$ under the action of the 
non-cyclic binary polyhedral groups. They also admit a description 
as  hypersurface singularities 
$$
(f(x,y,z)=0)\subseteq \CC^3,
$$
where 
\begin{equation}
\label{poly}
f(x,y,z)=
\begin{cases}
xy+z^k, \text{for singularities of type}~A_{k-1},~k\geq 2,\\
x^2y+y^{k-1}+z^{2}, \text{for singularities of the type}~ D_k,~k\geq 4,\\
x^4+y^3+z^2, \text{for singularities of the type}~E_6,\\
x^3y+y^3+z^2, \text{for singularities of the type}~E_7,\\
x^5+y^3+z^2, \text{for singularities of the type}~E_8.
\end{cases}
\end{equation}

\begin{defn}
A normal surface singularity is called of class $T$ if it is a 
rational double point or a cyclic quotient singularity 
the type $\displaystyle \frac{1}{dn^2}(1,dnm-1),$ for $d>0,~n\geq 1,$ 
and $(m,n)=1.$ 
\end{defn}

Using the natural sequence of abelian groups:
$$
1\ra\mu_{dn}\ra\mu_{dn^2}\ra \mu_n\ra 1,
$$
the second type of singularities can be described as the double quotient 
$$
\bigslant{\left({\CC^2/\frac1{dn}(1,-1)}\right)}{\mu_n},
$$
i.e. it is a quotient of an $A_{dn-1}-$ singularity.

If $n=1$ in case $(2)$ of the Proposition \ref{descsing}, the cyclic 
singularity is a rational double point of type $A_{d-1}.$ We treat the 
$A-$type singularities and their quotients concomitantly.

We consider  the hypersurface 
$\yy=(xy-z^{dn}=Q(z^n))\subseteq\CC^3\times \CC^d,$ 
where $\displaystyle Q(z)=\sum_{k=0}^{d-1}e_kz^{k}.$ 
It is convenient to introduce the polynomial 
$$
P(z)=z^d+Q(z)=\prod_{j=1}^{l}(z-a_j)^{k_j},
$$ 
where $a_1,\dots, a_l\in \CC$ are distinct, and the positive integers 
$k_j, j=1,\dots, l,$ satisfy $\displaystyle \sum_{j=1}^{l}k_j=d.$

We denote by $(x,y,z)$ and ${\bf e}=(e_0,\dots,e_{d-1})$ the  
linear coordinates on $\CC^3$ and $\CC^d,$ respectively. 
We define the action of the group $\mu_n$ on $\yy$ by:
\begin{equation}
\label{action}
\rho(x,y,z,e_0,\dots ,e_{d-1}):=(\rho x,\rho ^{-1}y,\rho ^{m}z, e_0,\dots ,e_{d-1}),
\end{equation}
where $\rho$ is a generator of $\mu_n.$
Let $\XX=\yy{/\mu_n}$ and $\phi:\XX\ra \CC^d$ the quotient of the projection 
$\yy \ra \CC^d.$ Let $X_0$ be the fiber $\phi^{-1}(0).$ Then $(X_0,0)$ is a 
singularity of the type $\displaystyle \frac{1}{dn^2}(1,dnm-1)$  and we have:

\begin{prop}\cite{manetti, ksb}
\label{defT} 
The map
$\phi:\XX\ra \CC^d$ is a $\QQ-$Gorenstein deformation of the 
cyclic singularity $(X_0,0)$ of type 
$\displaystyle \frac{1}{dn^2}(1,dnm-1).$
Moreover, every $\QQ-$Gorenstein deformation $\XX\ra \CC$ 
of a singularity $(X_0,0)$ of type 
$\displaystyle \frac{1}{dn^2}(1,dnm-1)$ 
is isomorphic to the pullback of $\phi$ for some germ of 
holomorphic map $(\CC,0)\ra(\CC ^d,0).$ 
\end{prop}

As in \cite{manetti}, given ${\bf e}\in \CC^d\setminus \{0\},$ 
we want the group $\mu_n,~n\geq2,$ to act freely on the fiber 
$Y_{\bf e}\subseteq\CC^3$ of the deformation $\yy\ra \CC^d.$ 
This condition is equivalent to the fact that  $\{0\}\in \CC^3$ 
lies only on the central fiber $Y_0,$ and it translates  into 
\begin{equation}
\label{man-cond}
a_j\neq 0,~~{\text{for every}}~~j=1,\dots,l.
\end{equation}
In the case $n=1,$ any fiber of the form $Y_{\bf e}=(xy=P(z))$ 
is biholomorphic to  a fiber of a deformation satisfying the above 
condition after a change of coordinates. We impose the condition 
(\ref{man-cond}) throughout this paper for any $n.$

The variety $X_{\bf e}=\phi^{-1}({\bf e})$ is the Milnor fiber of the 
$\QQ-$Gorenstein deformation if it is smooth. This translates into 
$l=d,$ and $k_j=1,~j=1,\dots, d.$

\subsection{Cyclic quotient singularities of class $T$}

In this section we construct a family of singular compactifications 
of a fiber of a $\QQ-$Gorenstein deformation of a singularity of 
the type  $\displaystyle \frac{1}{dn^2}(1,dnm-1).$ The compactifications 
are presented as hypersurfaces in appropriate weighted projective spaces.

Let 
\begin{equation}
\label{polynomial}
 P(z)=\prod_{j=1}^{l}(z-a_j)^{k_j},
\end{equation}
where $a_1,\dots, a_l\in \CC^{*}$ are distinct, and $k_j,~j=1,\dots, l,$ 
are positive integers with $\displaystyle \sum_{j=1}^lk_j=d.$ The variety  
$$
M=\left(xy=P(z^n)\right)/\mu_n\subseteq  \CC^3/\frac{1}{n}(1,-1,m)
$$
is a fiber of a $\QQ-$Gorenstein  deformation  of a 
$\displaystyle \frac{1}{dn^2}(1,dnm-1)-$singularity.

We define 
$$
{\overline M}=
\left(xy=w^{dk}P\left(\frac{z^n}{w^k}\right)=
\prod_{j=1}^{l}(z^n-a_jw^k)^{k_j}\right)\subseteq \PP(a,b,c,e),
$$ 
where we denoted by $[x:y:z:w]$ the homogeneous coordinates 
in the weighted projective space $\PP(a,b,c,e),$ and $k$ is a positive 
integer. The weights should satisfy the homogeneity conditions:
\begin{equation}
\label{hom}
a+b=dnc=dke.
\end{equation} 

We identify next sufficient conditions on the weights such that $M$ 
embeds into ${\overline M}$  as ${\overline M}\cap U_w.$

The standard affine coordinate chart 
$\displaystyle U_w=(w\neq 0)\subseteq \PP(a,b,c,e)$  
is isomorphic to $\CC_{(X,Y,Z)}^3/\frac1e(a,b,c).$
We require the action of $\mu_e$ to be equivalent to an 
action of the type $\frac1n(1,-1,m).$ This forces $e=n,$ 
and from the homogeneity condition we see that $k=c.$
Furthermore, there should exist $\rho\in \mu_n$ primitive 
root such that 
\begin{eqnarray}
\label{equiv-action}
 (1)&\rho=\xi^a;\notag \\
 (2)&\rho^{-1}=\xi^b; \notag \\
 (3)&\rho^m=\xi^c.
\end{eqnarray}

From the homogeneity condition we see that $a+b=0 \mod n,$ 
and so the conditions (1) and (2) are equivalent. Therefore, 
for any given $c\geq 0,$ the conditions $(1)-(3)$ are simultaneously 
satisfied if and only if 
\begin{equation}
\label{action-cond}
am=c \mod n.
\end{equation}

Let $u\in\{1,\dots,n-1\}$ be the unique integer such that $mu=1 \mod n.$ 
Then (\ref{action-cond}) is equivalent to
$$
a=cu \mod n.
$$

Finally, a condition we require is that the weighted projective space 
$\PP(a,b,c,n)$ is well-formed. From (\ref{hom}) and (\ref{action-cond}), 
we can see that this is equivalent to requiring that $\gcd(n,c)=1,$ 
which implies that $\gcd(a,n)=\gcd(b,n)=1.$ If $\gcd(a,c)=p\neq 1,$ 
then $\gcd(b,c)=p,$ and we write $a=pa', b=pb',$ and $c=pc',$ where 
$\gcd(a',c')=\gcd (b',c')=1.$ Notice that (\ref{hom}) yields $a'+b'=dnc',$ 
while from (\ref{action-cond}) we see that $a'=c'u \mod n.$ Moreover, 
we have an isomorphism $\PP(a,b,c,n)\simeq\PP(a',b',c',n)$ \cite{fle}. 
By replacing $(a,b,c)$ by $(a',b',c'),$ we can therefore assume 
that $\gcd(a,c)=1.$ Then the singularities of  $\PP(a,b,c,n)$ are in 
codimension at least $2.$ This simplifies the discussions regarding 
the singularities of the compactifications.

We summarize these requirements as
\begin{equation}
\label{div}
\gcd(c,n)=1~~\text{and}~~\gcd (a,c)=1.
\end{equation}

\begin{prop}
\label{compactific}
 Let $a,b$ and $c$ satisfying the conditions (\ref{hom}), (\ref{action-cond}) 
 and (\ref{div}), and let 
 $$
 {\ol M}=\left(xy=\prod_{j=1}^{l}(z^n-a_jw^c)^{k_j}\right)\subseteq \PP(a,b,c,n).
 $$ 
 We have 
\begin{itemize}
\item[ 1)] The variety $M$ embeds as a Zariski open subset in $\overline M.$ 

\item[ 2)] The singular points of $M$ are at most rational double points, 
of type $A_{k_j-1}$ for $j=1,\dots,l, k_j\geq2.$   

\item[ 3)] The singular points of $\ol M$ lying on ${\overline M}\setminus M$ 
are singularities of the types $\frac{1}{a}(c,n)$ and $\frac{1}{b}(c,n)$ at the 
points $R_1=[1:0:0:0]$ and $R_2=[0:1:0:0],$ respectively. 

\item[ 4)] The curve at infinity $C={\overline M}\setminus M$ is a smooth 
submanifold and a rational curve. Moreover, $C$ is an ample $\QQ-$Cartier 
divisor of $\ol M.$ With the induced complex structure $C$ has two singular 
points, $R_1, R_2,$ of type $\frac 1a(c), \frac 1b(c),$ respectively.

\item[ 5)] The anti-canonical divisor of ${\overline M}$ is ample, and as 
$\QQ-$Cartier divisors, we have 
\begin{equation}
\label{beta-condition}
-K_{\overline M}=\frac{c+n}{n}C.
\end{equation}
In particular, ${\overline M}$ is a log del Pezzo surface. 

\item[ 6)] The topological space $\overline M$ is simply connected 
and its second homology group has rank $d.$
\end{itemize}
\end{prop}
\begin{proof} $1)$ Notice that the condition (\ref{div}) implies that the 
weighted projective space is well-formed. The conditions (\ref{hom}) 
and (\ref{action-cond}) were imposed to ensure that $M$ embeds into 
$\ol M$ as ${\overline M}\cap U_w.$ 

\medskip

$2)$ As discussed, the chart 
$\displaystyle U_w=(w\neq 0)\subseteq\PP(a,b,c,n)$ 
is isomorphic to 
$\CC_{(X,Y,Z)}^3/\frac1n(a,b,c).$ 
In these coordinates 
$$
{\overline M}\cap U_w\simeq \left(XY=P(Z^n)\right)/\frac1n(a,b,c).
$$ 
From (\ref{div}) we find that the fixed point 
$(0,0,0)\notin {\ol M}\cap U_w$ 
is the only point of non-trivial isotropy. 

The singular points of the hypersurface 
$\displaystyle \left (XY=P(Z^n)\right)\subseteq\CC^3$ 
occur when the polynomial $P$ has multiple roots. 
We find that $S_j=[0:0:a_j^{\frac1n}:1]\in {\overline M}$ 
are singular points of type $A_{k_j-1}$ of  ${\overline M}\cap U_w,$ 
for any $j=1,\dots, l,$ such that $k_j\geq 2.$

\medskip

$3)$ We compute the singularities of $\overline M$ at infinity 
in the standard charts covering the weighted projective space 
$\PP(a,b,c,n).$

\smallskip

Let $U_x=(x\neq 0)\subseteq\PP(a,b,c,n).$ 
Then $U_x\simeq \CC_{(Y,Z,W)}^3/\frac1a(b,c,n).$ 
In these coordinates 
$$
{\overline M}\cap U_x\simeq 
\bigslant{\displaystyle{\left(Y=\prod_{j=1}^{l}(Z^n-a_jW^c)^{k_j}\right)}}{\frac1a(b,c,n)}.
$$ 
Since $a$ is relatively prime to $c$ and $n,$ the only 
point of non-trivial isotropy is the fixed point of the 
action of $\mu_a$ on 
$\CC^3,~(0,0,0)\in {\overline M}\cap U_x.$
As the hypersurface 
$\displaystyle \left(Y=\prod_{j=1}^{j}(Z^n-a_jW^c)^{k_j}\right)
\subseteq \CC^3$ 
is smooth, the only singular point of 
${\overline M}\cap U_x$ is $R_1=[1:0:0:0].$ 
Notice that the coordinates $(Z,W)$ parametrize 
${\overline M}\cap U_x,$ and so the point 
$R_1\in {\overline M}$ is a cyclic quotient singularity 
of the type $\frac{1}{a}(c,n).$

An analogous computation exhibits one more singular point of 
${\overline M}$ in the chart $U_y=(y\neq 0)\subseteq\PP(a,b,c,n).$ 
This point is $R_2=[0:1:0:0],$ a cyclic quotient singularity of the type 
$\frac{1}{b}(c,n).$ Moreover, in the chart  
$U_z=(z\neq 0)\subseteq\PP(a,b,c,n),$ as in the proof of part $1),$ 
we find no singular points on $\ol M\setminus M,$ and  
we recover the above singular points  
$S_j=[0:0:a_j^{\frac1n}:1]\in {\overline M}\cap U_z, k_j\geq 2.$

\medskip

$4)$ The curve at infinity $C={\overline M}\setminus M$ 
is the hyperplane section $(w=0),$ of weight $n.$ As a 
$\QQ$-Cartier divisor
$$
\OO_{\overline M}(C)=\OO_{\overline M}(n),
$$
where by $\OO_{\overline M}(1)$ we denote the restriction 
of the tautological sheaf  $\OO_{\PP(a,b,c,n)}(1)$ to 
$\overline M.$ 
In particular $C$ is an ample $\QQ-$Cartier divisor.

In the local charts 
$\ol M\cap U_x\simeq \CC^2_{(Z,W)}/\frac1a(c,n)$ and 
$\ol M\cap U_y\simeq \CC^2_{(Z,W)}/\frac1b(c,n),$ 
the curve at infinity $C$ is given by the $(W=0).$ 
Hence $C\cap \ol M\cap U_x\simeq \CC/ \frac1a(c),$ and 
$C\cap \ol M\cap U_y\simeq \CC/\frac1b(c),$ 
and  $C$ is smooth submanifold of $\ol M.$
 Moreover, since in the local chart  $\ol M\cap U_x$ 
 the curve $C$ corresponds to the $Z-$axis, and 
 $C$ is a one point compactification of 
 $C\cap \ol M\cap U_x,$ then it must be rational.

\medskip

$5)$ Since the weighted projective space $\PP(a,b,c,n)$ 
is well-formed and the hypersurface $\ol M$ does not 
contain its singular lines (if any), then the adjunction 
formula holds. The canonical divisor of $\ol M$ is 
$$
K_{\overline M}=(K_{\PP(a,b,c,n)}+{\overline M})|_{\overline M}
=\OO_{\overline M}(dnc-a-b-c-n).
$$
Using the homogeneity condition (\ref{hom}) we obtain
$$
K_{\overline M}=\OO_{\overline M}(-c-n)=-\frac{c+n}{n}C.
$$
In particular, we see that $\ol M$ is a log del Pezzo surface.

$6)$ To study the topological properties of $\ol M,$ we have 
to consider two cases. In the first case, when $n=1$ the manifold 
$M$ is a smooth deformation of a singularity of type $A_{d-1}.$ 
As a consequence, it is simply connected and the rank of its 
second homology is $d-1.$ The compactification is obtained by 
adding a rational curve at infinity, hence by Van Kampen's theorem, 
the topological space $\ol M$ is simply connected. 
Moreover, with respect to the decomposition 
$\displaystyle \ol M=M\cup Nbhd(C),$ the intersection 
$\displaystyle M\cap Nbhd(C)$ is homotopic to the lens space 
$L_d(1,-1)$ whose first Betti number is  
$\displaystyle b_1\left(L_d\left(1,-1\right)\right)=0.$ 
Thus, by the Mayer-Vietoris sequence the second  
Betti number is $b_2(\ol M)=d.$

In the second case, when $n\geq2,$ the manifold $M$ is obtained 
by taking the quotient of a deformation of the $A_{dn-1}$ singularity 
by a free $\mu_n-$action. Hence the fundamental group of $M$ is 
$\pi_1(M)=\ZZ/n\ZZ.$ Moreover, the fundamental group of a 
neighborhood of infinity is 
$\displaystyle \pi_1(L_{dn^2}(1,dnm-1))=\ZZ/{dn^2}\ZZ,$ 
and there is a natural surjection 
$\pi_1((L_{dn^2}(1,dnm-1))=\ZZ/{dn^2}\ZZ\to \pi_1(M)=\ZZ/n\ZZ.$ 
The Euler characteristic  of $M$ is 
$$
\chi(M)=\frac1n\chi(A_{dn-1})=\frac1n dn=d,
$$ 
hence the second Betti number of $M$ equals $d-1.$ 
Van Kampen's theorem and the Mayer-Vietoris sequence for 
$\displaystyle \ol M=M\cup Nbhd(C)$ imply that the space $\ol M$ 
is simply connected, and its second Betti number is $d.$
\end{proof}

\begin{rmk}
For a given integer $c\geq 0,$ with $\gcd(c,n)=1,$ it is easy to see that 
the conditions (\ref{equiv-action}) will be satisfied by at least one pair 
$(a,b)$ of non-negative integers. As a consequence, when taking 
different values of $c,$ the manifold $M$ can be embedded in 
infinitely many log del Pezzo surfaces.
\end{rmk}

\begin{examples} The embeddings of $M$ in log del Pezzo surfaces in 
Theorem \ref{compactification} are indexed by the set of weights $(a,b,c).$ 
We list below some interesting cases:
\begin{itemize}
\item[ 1)] an immediate choice of weights induced by the action 
$\mu_n$ are for $c=m$ and hence: 
$$
(a,b,c)=(1+knm, (d-k)nm-1,m), ~k\in\{0,\dots,d-1\}.
$$

\item[ 2)] another normalization can be obtained for $c=1,$ and hence 
$$
(a,b,c)=(u+kn, (d-k)n-u,1), ~k\in\{0,\dots,d-1\}.
$$ 
The case $d=1,~k=0$ also appears in \cite{k-bmy}.
\end{itemize}
\end{examples}
\begin{rmk}
 When $n=1$, we compactify the singularity $A_{d-1}$ 
and its deformations  
in $\PP(a,dc-a,c,1).$  A special case, when $c=1,~d$ even and 
$a=b$ appears in Saito \cite{saito}. We generalize Saito's compactification 
to infinitely many sets of weights and to arbitrary complex structures. 
The $A-$case also appears in \cite{craig}.
\end{rmk}

\subsection{A second compactification of a deformation of a 
$\frac 1{dn^2}(1,dnm-1)$ singularity, for $n>1$}
\label{secondconstr}

Let 
$$
M=\left(xy=P(z^n)\right)/\mu_n\subseteq  \CC^3/\frac{1}{n}(1,-1,m),
$$ 
where  $\displaystyle P(z)=\prod_{j=1}^{l}(z-a_j)^{k_j},$ with $a_1,\dots, a_l\in \CC^{*}$ 
distinct, $k_j,~j=1,\dots, l,$  positive integers with 
$\displaystyle \sum_{j=1}^lk_j=d$ be the generic fiber of a deformation 
as in (\ref{polynomial}). 
Let $N=\left(xy=P(z^n)\right)\subseteq  \CC^3$ and 
$\ol N=\left(xy=w^{dnc}P\left(\left({\frac {z}{w^c}}\right)^n\right)\right)$ 
be its compactification in $\PP(a,b,c,1)$ with $\gcd(a,b,c)=1.$

We construct a second compactification of $M$ as a $\mu_n-$quotient of $\ol N.$ 
On $U_w=(w\neq0)\simeq \CC^3\subseteq \PP(a, b,c,1)$ the action must be 
of the form $\frac1n(1,-1,m).$ 
We consider its extension to $\PP(a,b,c,1)$ defined as 
\begin{equation}
\label{quot}
 \rho[x:y:z:w]=[\rho x:\rho^{-1}y:\rho^mz:w] ,\rho\in\mu_n
\end{equation}

Let $\wt M$ be the quotient $\ol N/\mu_n,$ and $\pi:\ol N\to \wt M$ 
the holomorphic quotient map.

\begin{prop}
\label{compact2}
Let $\wt M$ be as above, and $C=\wt M\setminus M$ 
the divisor at infinity. Then we have:
\begin{itemize}
\item[ 1)] The variety $M$ embeds as a Zariski open subset in 
$\wt M,$ and the singular points of $M$ are at most rational double 
points of type $A_{k_j-1}$ for $j=1,\dots,l, k_j\geq2.$

\item[ 2)] Let $k=\gcd(n,am-c).$ When $k=1$  the action of $\mu_n$ 
on $\ol N$ is semi-free and has only two fixed points  $R_1=[1:0:0:0]$ 
and $R_2=[0:1:0:0].$ Otherwise, the map $\pi:\ol N\to \wt M$ is a 
branched covering of order $k,$ with branch locus $C_w=(w=0).$

\item[ 3)]  If $k=1$ the singular points of $\wt M$ lying on $C$ 
are singularities of the types $\frac{1}{an}(c-am,1)$ and 
$\frac{1}{bn}(c+bm,1)$ at the points $\pi(R_1)$ and $\pi(R_2),$ 
respectively. If $k>1,$ then the  singular points of $\wt M$ lying on $C$ 
are singularities of the types $\frac{1}{a\frac nk}(\frac{c-am}k,1)$ and 
$\frac{1}{b\frac nk}(\frac{c+bm}k,1)$ at the points $\pi(R_1)$ and $\pi(R_2),$ 
respectively.

\item[ 4)] The curve at infinity is a smooth suborbifold, rational, 
ample $\QQ-$Cartier divisor of $\wt M.$  With the induced complex 
structure, $C$ has two singular points, $\pi(R_1), \pi(R_2),$ of order 
$\frac {an}k, \frac {bn}k,$ respectively.

\item[ 5)] The anti-canonical divisor of ${\wt M}$ is ample, and as 
$\QQ-$Cartier divisors, we have 
\begin{equation}
\label{beta-condition2}
-K_{\wt M}=\frac{k+c}kC. 
\end{equation}
In particular, ${\wt M}$ is a log del Pezzo surface.

\item[ 6)] The topological space $\wt M$ is simply connected and 
its second homology group has rank $d.$
\end{itemize}
\end{prop}
\begin{proof}
$1)$ The action of $\mu_n$ on $N=\ol N \cap U_w$ is just the restriction 
of the free action $\frac1n(1,-1,m)$ on $\CC^3\setminus \{0\}.$ 
Hence $N\to M=N/\mu_n$ is the universal covering, and the rest is immediate.

$2)$ and $3)$ We need to describe the action $\mu_n$ in a neighborhood 
of the curve at infinity. On 
$
\ol N\cap U_x\subset U_x\simeq\CC^3_{(Y,Z,W)}/\frac1a(b,c,1)
$ 
the manifold is given by: 
$$
\bigslant{\left(Y={\displaystyle \prod_{j=1}^{l}}(Z^n-a_jW^c)^{k_j}\right)}{\frac1a(b,c,1)}
$$ 
and it is locally parametrized by 
$\CC^2_{(Z,W)}/\frac1a (c,1).$ 
In the chart $U_x$ the action  $\mu_n$ is of the form 
\begin{equation}
\label{act-Ux}
\rho(Y,Z,W)=(\rho^{-1-\frac ba}Y,\rho^{m-\frac ca}Z, \rho^{-\frac 1a}W)=
(\zeta^{a+b}Y,\zeta^{c-am}Z,\zeta W)   
\end{equation} 
for  some $\zeta=\rho^{-\frac1a}.$ 
By a fractional power we assume an arbitrary and consistent 
choice of $\zeta.$ This is well defined since we work on 
$
\CC^3/\frac1a(b,c,1).
$ 
The local chart of $(\ol N\cap U_x)/\mu_n$ is of the form 
$
\left(\CC^2_{(Z,W)}/\frac1a(c,1)\right)/\frac1n(m-\frac ca,-\frac1a)
$ 
which is isomorphic to 
$
\CC^2_{(Z,W)}/\frac1{an}(c-am,1).
$
Note that $\gcd(a,c-am)=1,$ as we choose $a,b,c$ such that 
$\gcd(a,b,c)=1$ and $a+b=dnc.$ 
The quotient $\CC^2_{(Z,W)}/\frac1{an}(c-am,1)$ 
is a branched cover if $\gcd(n,c-am)=k\neq1.$ 
In this case we have new complex coordinates $(Z,U)=(Z,W^k)$ 
and the new local chart is given by 
$\CC^2_{(Z,U)}/\frac1{a\frac nk}(\frac{c-am}k,1).$ 
This gives us the type of the singularity at $\pi(R_1).$

The computation in the chart 
$
U_y\simeq\CC^3(X,Z,W)/\frac1b(a,c,1)
$ is similar, and we omit it.
We obtain an induced complex chart on $\wt M$ of the form
$$
\bigslant{\left(\CC^2_{(Z,W)}/\frac1b(c,1)\right)}{ \frac1n(m+\frac cb,\frac1b)}
\simeq\CC^2_{(Z,W)}/\frac1{bn}(c+bm,1).
$$
As $a+b=dnc,$ we have that $\gcd(n,c-am)=\gcd(n,c+bm),$ 
and we regain the result regarding the branched covering. 
The type of the singularity at $\pi(R_2)$ follows from this.

$4)$ The curve at infinity is a quotient of the curve 
$C_w=(w=0)\subset\ol N\subset\PP(a, b,c,1),$ 
and is given in local coordinates by $W=0,$ or $U=0.$ 
Hence $C$ is a smooth, ample $\QQ-$divisor, 
and an orbifold rational curve with two singularities 
of order $a\frac nk,$ and $b\frac nk,$ at 
$\pi(R_1), \pi(R_2),$ respectively.

$6)$ The topological space $\wt M$ is simply connected 
as it is the compactification of $M$ by an orbifold rational curve, 
and the same argument from Proposition \ref{compactific} 
(part 6)) applies here.

$5)$ For $k$ as above, we must have that $a$ and $k$ 
are relatively prime, otherwise we obtain that the 
$\gcd(a,b,c)$ is divisible by $\gcd(a,k)\neq1.$ 
Hence, there are integers $r,s$ such that $ar+ks=1.$ 
Then the action of $\mu_k=\langle\alpha\rangle$ 
as a subgroup of $\mu_n$ on $\PP(a, b,c,1)$ is of the form
\begin{eqnarray*}
\alpha[x:y:z:w]&=&[\alpha x:\alpha^{-1}y:\alpha^mz:w] \\ \notag
&=&[(\alpha^{-r})^a\alpha x:(\alpha^{-r})^b\alpha^{-1}y:(\alpha^{-r})^c\alpha^{m}z:(\alpha^{-r})w] \\
&=&[\alpha^{-ar+1}x: \alpha^{-1-rb}y:\alpha^{m-rc}z:\alpha^{-r}w]\\
&=&[\alpha^{ks}x:\alpha^{-1+ar-dncr}y:\alpha^{m-rc}z:\alpha^{-r}w]\\
&=&[x:\alpha^{-ks-kd(\frac nk)cr}y:\alpha^{m-rc}z:\alpha^{-r}w]\\
&=&[x:y:\alpha^{m-rc}z:\alpha^{-r}w]\\
&=&[x:y:z:\alpha^{-r}w].
\end{eqnarray*}

To justify the last equality, we must show that $\alpha^{m-rc}=1.$ 
As  $\alpha$ is a $k-$root of unity and $\gcd(a,k)=1,$ it is enough 
to prove that $(\alpha^{m-rc})^a=1.$ A quick computation yields 
$(\alpha^{m-rc})^a=\alpha^{am-cra}=\alpha^{am-c+c(1-ar)}=1.$ 
 
 Hence $\PP(a,b,c,1)/\mu_k$ is isomorphic to $\PP(a,b,c,k),$  by 
 $[x:y:z:w]\mapsto[x:y:z:u],$ $u=w^k.$ 
 If we denote by $n'=\frac nk,$ then 
 $$\PP(a,b,c,1)/\mu_n\simeq\PP(a,b,c,k)/\mu_{n'},$$
 where the $\mu_{n'}$ is the induced action on $\PP(a,b,c,k).$ 
On our hypersurface the action $\mu_{n'}$ is semi-free  with two 
fixed points corresponding to $R_1$ and $R_2.$
 The map $\pi:\PP(a,b,c,1)\to \PP(a,b,c,1)/\mu_n$ can be written 
 as the composition of two projections:
 $$
 \PP(a,b,c,1) \xrightarrow{\pi_1}\PP(a,b,c,k)\xrightarrow{\pi_2}\PP(a,b,c,1)/\mu_n.
 $$
Let $\ol N'=\pi_1(\ol N),$ and $C_u=\pi_1(C_w)=(u=0)\subset\PP(a,b,c,k).$ 
Then, $-K_{\ol N'}=\frac{k+c}k C_u$ as in Proposition \ref{compactific}.
 The manifold $\wt M$ is the $\mu_{n'}-$quotient of $\ol N'$ 
 by a semi-free action, hence
 $$
 \pi_2^*(-K_{\wt M})=-K_{\ol N'}=\frac{k+c}k C_u=\frac{k+c}k \pi_2^*(C).
 $$ 
 Since $\wt M$ is simply connected, there are no torsion line bundles and 
 this implies 
 $-K_{\wt M}=\frac{k+c}k C,$ as $\QQ-$Cartier divisors.
\end{proof}

\subsection{Non-cyclic quotient singularities of class $T$}

In this section we discuss a family of singular compactifications of a 
fiber of the universal deformation of  singularities of the types 
$D_k, E_6, E_7$ and $E_8.$ These are also hypersurface singularities, 
and this allows us to exhibit the compactifications as hypersurfaces 
in appropriate weighted projective spaces. The construction appears 
in Saito's paper \cite{saito}, although his results are focused on the 
Coxeter transformation associated to the groups. We extend it to an 
arbitrary fiber of the universal deformation space of a rational double point.

Let $f$ be one of the polynomials for the $D-,E-$type singularities 
(\ref{poly}), and let $g_i\in \CC[x,y,z], i=1\dots,k$ be monomials 
yielding a basis of 
$$
\CC[x,y,z]/\langle f, \partial f/\partial x, \partial f/\partial y,\partial f/\partial z\rangle.
$$

The fiber of the universal deformation of such a rational double 
point singularity is the affine variety 
$$
M=\left(f=\sum_{i=1}^ka_ig_i\right)\subseteq \CC^3,
$$ 
where $a_i\in \CC$ are fixed  \cite{artin}. Notice that for $a_i's$ 
general enough, the surface $M$ is smooth. In general, it is known 
that $M$ has at most rational double points singularities. 
If $a_i=0,~i=1,\dots,k,$ we recover the central fiber which is a 
singularity of the type $\CC^2/G,$ where $G\subseteq SU(2).$ 
For convenience, let 
$
\displaystyle h=f-\sum_{i=1}^ka_ig_i\in \CC[x,y,z].
$

As before we would like to embed $M$ in a weighted projective space 
$\PP(a,b,c,1)$ as $\ol M\cap U_w,$ where $U_w$ is the standard chart 
$(w\neq 0)$ and $\ol M$ is a hypersurface given by a suitable 
quasi-homogenization of the polynomial $h:$ 
$$
\ol M=\left(w^Nh\left(\frac{x}{w^a},\frac{y}{w^b},\frac{z}{w^c}\right)=0\right) 
\subseteq \PP(a,b,c,1),
$$
for some positive integers $a,b,c, $ and $N.$ A brief inspection of the
 polynomials (\ref{poly}) yields weights defined uniquely up to a 
common factor. The ambiguity is eliminated by the divisibility condition 
$\gcd(a,b,c)=1,$ which ensures that the ambient weighted projective space 
is well-formed. We find
\begin{equation}
\label{weights-def-non-cyclic}
(a,b,c)=
\begin{cases}
(k-2,2,k-1), \text{for singularities of the type}~D_k\\
(3,4,6), \text{for singularities of the type}~E_6\\
(4,6,9), \text{for singularities of the type}~E_7\\
(6,10,15), \text{for singularities of the type}~E_8.
\end{cases}
\end{equation}

For the rest of this section, we  consider $h$ a quasi-homogeneous 
polynomial as above, and $(a,b,c)$ the corresponding weights. 
With these choices, the weighted degree of $\ol M$ in $\PP(a,b,c,1)$ 
is $N=a+b+c-1.$

\begin{prop}
\label{comp-def-non-cyclic}
Let 
$$
\ol M=\left(w^Nh\left(\frac{x}{w^a},\frac{y}{w^b},\frac{z}{w^c}\right)=0\right) 
\subseteq \PP(a,b,c,1),
$$
where $(a,b,c)$ as in (\ref{weights-def-non-cyclic}). Then 
\begin{itemize}
\item[ 1)] $M$  embeds as a Zariski open subset in $\ol M.$  

\item[ 2)] The singular points of $\ol M$ lying in 
${\overline M}\setminus M$ are as follows: 
\begin{itemize}
\item[ i)] {\bf Case $D_{k}$:} two singularities of type $\frac12(1,1),$
 and one of type $\frac{1}{k-2}(1,1).$

\item[ ii)]  {\bf Case $E_{6}$:} two singularities of type $\frac13(1,1),$ 
and one of the type $\frac12(1,1).$

\item[ iii)]  {\bf Case $E_{7}$:} three singularities of type 
$\frac12(1,1), ~ \frac13(1,1),$ and  $\frac14(1,1),$ respectively.
\item[ iv)] {\bf Case $E_{8}$:} three singularities of type 
$\frac12(1,1),~ \frac13(1,1),$ and $\frac15(1,1),$ respectively.

\end{itemize}

\item[ 3)] The complement $C={\overline M}\setminus M$ is a smooth, 
rational, ample $\QQ-$Cartier divisor of $\ol M.$ Moreover, with the 
induced complex structure it has three singular points of orders 
$(2,2,k-2), (2,3,3), (2,3,4), (2,3,5)$ for the $D_k, E_{6,7,8}-$cases, 
respectively.

\item[ 4)] The anti-canonical divisor of ${\overline M}$ is 
\begin{equation}
\label{beta-cond-def-non-cyclic}
-K_{\overline M}=2C.
\end{equation}
In particular, ${\overline M}$ is a log del Pezzo surface. 

\item[ 5)] The variety $\overline M$ is simply connected and 
its second homology group has rank $k+1$ for a singularity 
of type $D_k,$ and $n+1$ for a singularity of type $E_n, ~n=6,7,8.$
\end{itemize}
\end{prop}

\begin{proof} The proofs of $\text{1),~2)},$ and $\text{3)}$  for $h=f-1$ 
can be found in \cite{saito}. One can easily check that the singularities 
at infinity remain the same for any deformation $h.$ 
The canonical class of $\ol M$ follows again from the adjunction formula:
$$
K_{\overline M}=(K_{\PP(a,b,c,1)}+{\overline M})|_{\overline M}=
\OO_{\overline M}(N-a-b-c-1)=\OO_{\overline M}(-2)=-2C.
$$
The proof of $5)$ is as in Proposition \ref{compactific}.6.
\end{proof}

Notice that in particular we obtained a compactification of a 
rational double point singularity with the properties stated in 
Proposition \ref{comp-def-non-cyclic}.

\subsection{Conclusions}

If $M$ is a fiber of a $\QQ-$Gorenstein deformation of a singularity 
of class $T,$ Propositions \ref{compactific}, \ref{compact2} and 
\ref{comp-def-non-cyclic} provide compactifications with the properties 
summarized in Theorem \ref{compactification}.

In general, $M$ might admit rational double points as singularities. 
In this case, we consider the minimal resolution $N$ of $M,$ and 
this gives us a special complex structure on the Milnor fiber of the singularity.
We have the following:

\begin{cor}
\label{res-fiber}
The minimal resolution $N$ of a fiber $M$ of a $\QQ-$Gorenstein 
deformation of a singularity of class $T$ embeds into a variety $\ol N$ 
as the complement of a smooth rational curve, which is a rational 
multiple of the anticanonical divisor. The singularities along the 
divisor at infinity are all isolated finite cyclic quotients. 
Moreover, if $M$ is associated to a  finite cyclic quotient 
singularity then there are infinitely many minimal compactifications 
with the above properties. 
\end{cor}
\begin{proof} Let $p:N\to M$ be the minimal resolution of $M.$ 
Correspondingly, let $p:\ol N\to \ol M$ denote its extension to the 
compactification. Since the singular points of $M$ are at most 
rational double points 
\begin{equation}
\label{can-res}
K_{\ol N}=p^*K_{\ol M}=-\beta C,
\end{equation} 
for some $\beta>1.$ 
The singularities of  $\ol N$ are only along the divisor at infinity as 
described in Propositions \ref{compactific} and \ref{comp-def-non-cyclic}. 
\end{proof}

Notice that if $M$ is singular, then $\ol N$ is no longer a 
log del Pezzo surface, as it contains the $(-2)-$curves 
introduced when resolving the singularities. Moreover, 
the divisor at infinity $C$ is only almost ample (see Definition 
\ref{ample-adm} below).

\begin{rmk}
Let  $M_0=\CC^2/\dfrac1{dn^2}(1,dnm-1), n\neq1,$ and 
$N$ its minimal resolution. Then $ N$ has non-trivial, even 
non-torsion as $\pi_1(N)=0$, canonical divisor, and this 
implies that $N$ does not admit Ricci-flat K\"ahler metrics. 
\end{rmk}

\section{The relation with Tian-Yau's Ricci-flat K\"ahler metrics}
\label{T-Ysec}

In this section we recall the relevant results of Tian-Yau \cite{ty2} 
and Bando-Kasue-Nakajima \cite{bkn}. Both Tian-Yau \cite{ty2} 
and Bando-Kasue-Nakajima \cite{bkn}  proved more general results, 
but we restrict the presentation to the ALE Ricci-flat case 
in complex dimension two which suffices for our purpose. 
We conclude by proving Theorem \ref{metrics}, and its corollary.

Tian and Yau construct \cite{ty2} complete Ricci-flat K\"ahler metrics 
on the complement of a divisor on compact K\"ahler orbifolds 
satisfying certain conditions.

\begin{defn}
\label{ample-adm} 
Let $D$ be a divisor in the K\"ahler orbifold ${\overline M}$ 
of complex dimension 2. Then
\begin{itemize}
\item[ (i)] $D$ is almost ample if there exists an integer 
$m > 0$ such that a basis of $H^0({\overline M},\OO(mD))$ 
gives a morphism from ${\overline M}$ into some projective space 
$\PP^N$ which is a biholomorphism in a neighborhood of $D.$

\item[ (ii)] $D$ is admissible if $\Sing({\overline M})\subseteq D,~ D$ 
is smooth in ${\overline M}\setminus\Sing({\overline M}),$ and if 
$\pi_x: {\widetilde U}_x\rightarrow U_x$ is the local uniformization at 
$x\in\Sing(\ol M)$ with ${\widetilde U}_x\subseteq \CC^2,$ then
$\pi_x^{-1}(D)$ is smooth in ${\widetilde U}_x.$
\end{itemize}
\end{defn}

For surfaces, Tian-Yau proved:

\begin{thm}[\cite{ty2}]
\label{ty-cond} Let $\overline M$ be a compact K\"ahler orbifold of complex 
dimension $2.$ Let $D$ be an almost ample, admissible divisor in 
$\overline M,$ such that 
\begin{equation}
\label{multiplicity}
-K_{\overline M} = \beta D, {\text ~for ~some}~ \beta>1.
\end{equation}
Suppose that $D$ with the induced complex structure admits a positive 
K\"ahler-Einstein metric, then $M={\overline M} \setminus D$ admits a 
complete Ricci-flat K\"ahler metric $g$ in every
K\"ahler class in $H^2_c(M,\RR).$

Moreover, if we denote by ${\mathcal R}(g)$ the curvature tensor of 
$g$ and by $r$ the distance function on M from some fixed point with 
respect to $g,$ then ${\mathcal R}(g)$ decays at the order of at least 
$r^{-3}$ with respect to the $g-$norm whenever  $D$ is biholomorphic 
to $\PP^1$ and
\begin{equation}\label{adjunction}
\OO_D(D)= \frac{2}{\beta-1}\OO_{\PP^1}(1).
\end{equation}
Furthermore the metric $g$ has euclidean volume 
growth.
\end{thm}

The Tian-Yau Theorem also requires that the divisor $D$ is neat 
(Definition 1.1(i) \cite{ty2}). Tian and Yau remark that this condition 
is probably superfluous in general. In complex dimension two this 
condition is automatically implied by the almost ampleness.

In their paper, Tian-Yau do not emphasize the cohomology 
class of the K\"ahler metric, but upon a close inspection 
of \cite{ty2}, the metrics can be constructed in an arbitrary compactly supported K\"ahler 
class. This issue is discussed in Joyce 
\cite{joyce} and Van Coevering \cite{craig}.

The metrics constructed using the Tian-Yau result are not a priori ALE, 
where by an ALE $4-$manifold we understand:
\begin{defn}
\label{ALE}
Let $G$ be a finite subgroup of $SO(4)$ acting freely on 
$\RR^4\setminus \{0\},$ and let $h_0$ be the Euclidean metric on 
$\RR^4/G.$ We say that the manifold $(M^4, g)$ is an ALE  manifold 
asymptotic to $\RR^4/G$ if there exist a compact subset $K\subseteq M$ 
and a map $\pi:M\setminus K\to \RR^4/G$ that is a diffeomorphism 
between $M\setminus K$ and the subset $\{z\in\RR^4/G ~| ~r(z)>R\}$  
for some fixed $R>0,$ such that 
\begin{equation}
\label{approx}
\nabla^k(\pi_*(g)-h_0)=O(r^{-4-k})~\text{for all}  ~k\geq 0.
\end{equation}
\end{defn}

If the metric is K\"ahler, then the group $G$ is a subgroup of $U(2),$ 
and the diffeomorphism $\pi$ identifies $M\setminus K$ with a subset 
$\CC^2/G.$ This identification is not a biholomorphism in general.

A remarkable result of Bando-Kasue-Nakajima proves that there exists 
a good asymptotic coordinate system under curvature decay conditions:

\begin{thm}[\cite{bkn}]
\label{bkn-thm}
Let $(M, g)$ be a Ricci-flat K\"ahler surface with
\begin{itemize}
\item[ 1)] $\displaystyle {\text{Vol}} B(p; r)\geq Cr^4$ for some $p\in M, C>0,$
\item[ 2)] $\displaystyle \int_M |{\mathcal R}(g)|^{2}dV_g<\infty.$
\end{itemize}
Then 
$(M, g)$ 
is ALE. 
\end{thm}
Here $B(p; r)\subseteq M$ denotes the ball or radius $r$ in $(M,g)$ 
centered at the point $p\in M.$

\begin{proof}[Proof of Theorem \ref{metrics}]
\label{proofA}
In Corollary \ref{res-fiber}, we proved that the complex surface 
$(M, J)$ admits a compactification to a variety $(\ol M, J)$ 
with at most  three finite cyclic quotient singular points along 
the divisor at infinity $C,$ and no other singular points. 
We showed that $C$ is an almost ample, admissible, smooth 
rational curve. Moreover, the numerical conditions (\ref{multiplicity}), 
(\ref{adjunction}) of the Tian-Yau construction  are also satisfied, 
as it can be seen from (\ref{can-res}) and the adjunction formula. 
The complex structure of the divisor at infinity was studied in 
Proposition \ref{compactific}.(4), Proposition \ref{compact2}.(4), 
and Proposition \ref{comp-def-non-cyclic}.(3).

The important condition in the Tian-Yau result is that the divisor 
$D$ admits a positive K\"ahler-Einstein metric with respect to the 
complex orbifold structure induced by $\ol M.$ In the case of $S^2$ 
with three orbifold points, the orders of the singularities are $(2,2,k-2)$ 
in the $D_k$ case, $(2,3,3), (2,3,4), (2,3,5)$ in the $E_{6,7,8}$ cases, 
which are exactly the orbifolds which are global quotients of $S^2$ by 
the dihedral and polyhedral groups. In particular, they all admit positive 
Einstein metrics. The metrics are K\"ahler-Einstein as we are in 
oriented surfaces and real dimension $2.$

In the case of $S^2$ with two orbifold points a well known result of 
Troyanov  \cite{tr} tells us that the Einstein metric  exists if and only if 
the two orbifold points have equal order, and in this case we have the 
quotient of the sphere with the canonical metric. The order of the 
singularities in the compactifications obtained in Proposition 
\ref{compactific}.(4) are $a$ and $b.$ This implies $a=b.$ As 
$a+b=dnc$ and $\gcd(a,b,c)=1$ this implies that $c=1$ when 
$dn$ is even, and $c=2$ when $dn$ is odd. Moreover, when the 
conditions (\ref{equiv-action}) are satisfied, we must have 
$\rho=\pm 1,$ so $n=1$ or $2.$ In particular, the first construction 
yields compactifications for singularities of type: $A_{d-1}$ in 
$\PP(d,d,2,1)$ when $d$ is odd, or in $\PP(\frac d2, \frac d2,1,1)$ 
when $d$ is even, and for singularities of type $\frac1{4d}(1,2d-1),$ 
when $d$ odd, in $\PP(d,d,1,2).$

For all the other cyclic singularities we use the second 
construction from Section \ref{secondconstr}. In general, 
we obtain that the deformation of a singularity of type 
$\frac1{dn^2}(1,dnm-1)$ compactifies as in Proposition \ref{compact2} 
in $\PP(dn,dn,2,1)/\mu_n$ when $dn$ is odd, or in 
$\PP(\frac{dn}2,\frac{dn}2,1,1)/\mu_n$ when $dn$ is even. 
Notice that for $n=2$ the two constructions give in fact the same 
compactification.

Therefore, in all possible cases, there exists a unique construction 
among our compactifications which satisfies the hypothesis of the 
Tian-Yau Theorem. Hence, the complement of the divisor at infinity, 
$\ol M\setminus C,$ admits a complete Ricci-flat K\"ahler metric 
in any K\"ahler class. This metric must be ALE by Theorem \ref{bkn-thm}. 
Hence, by the uniqueness part of Theorem \ref{classif}, the Tian-Yau 
method rediscovers the unique metric $g$ in the K\"ahler class  $[\omega_g].$
\end{proof}

\section{Smooth compactifications}

In this section we discuss the smoothness of our compactifications, 
and prove Proposition \ref{smooth-comp}. 
\begin{lem} 
\label{smoothcase}
Let $(\ol M, [D])$ be a pair consisting of smooth complex surface 
equipped with a smooth ample divisor $[D]$ homeomorphic to $S^2.$ 
If there exists $\beta> 1$ such that $-K_{\ol M}=\beta [D]$ then $(\ol M, D)$ 
is one of the following:
\begin{itemize} 
\item[ 1)] $(\PP^1\times\PP^1, [Diagonal]),$ 

\item[ 2)] $(\PP^2, [Line]),$  

\item[ 3)] $(\PP^2, [Conic]).$ 
\end{itemize}
\end{lem}

\begin{proof}
The condition $-K_{\ol M}=\beta [D]$ with $\beta>1$ and $[D]$ ample 
implies that $-K_{\ol M}$ is ample, that is $\ol M$ is a del Pezzo surface. 
In particular, from the classification of del Pezzo surfaces we know that 
$K_{\ol M}^2\in\{1,2,\dots,9\}.$ On the other hand, from the adjunction 
formula for the smooth divisor $[D]$ and using again that  
$-K_{\ol M}=\beta [D],$ we infer that
$$
K_{\ol M}^2=\frac{2\beta^2}{\beta-1}.
$$
This is possible only when $K_{\ol M}^2=8,$ in which case 
$\beta=2,$ or $K_{\ol M}^2=9,$ in which case either $\beta=3$ or 
$\beta=3/2.$ Using again the classification of del Pezzo surfaces, 
in the first case $(\ol M,[D])$ must be $(\PP^1\times\PP^1, [Diagonal]),$ 
while in the second case $\ol M$ is $\bcp^2$ and $[D]$ is either a line 
or a smooth conic.
\end{proof}

\begin{prop}
\label{smooth-comp}
The only cases in our constructions when the compactification 
$\ol M$ is a smooth surface is when $M$ is $\CC^2,$ an $A_1-$type 
manifold, or the $\ZZ_2$ quotient of an $A_1,$ and in these cases 
the compactifications are $ (\PP^2,[Line])$ $(\PP^1\times\PP^1, [Diagonal]),$ 
or $(\PP^2, [Conic])$ respectively.
\end{prop}
\begin{proof}
The singular locus of the compactifications discussed is studied in 
Propositions \ref{compactific}, \ref{compact2} and \ref{comp-def-non-cyclic}.

Proposition \ref{compactific} is used to compactify  deformations of 
$A-$type singularities, in which case $n=1.$  They are smooth if and only if 
$a=b=1.$ When $n=1,$ then the condition (\ref{hom}) forces either 
$d=1$ and $c=2$ or $d=2$ and $c=1.$ In the case when $d=2$ and $c=1$ 
we deform a $A_1-$singularity, and the compactification we found is 
$$
\ol M=\left(xy=w^2P\left(\frac{z}{w}\right)\right)\subset \PP(1,1,1,1)
$$ 
In this case, $\ol M$ is a smooth quadric, and hence isomorphic to 
$\PP^1\times \PP^1,$ and the divisor at infinity is the diagonal.

In the case when $d=1$ and $c=2,$ we have in fact $M\simeq\CC^2$. 
Our compactification  in this case is
$$
\ol M=\left(xy=w^2P\left(\frac{z}{w^2}\right)\right)\subset \PP(1,1,2,1)
$$ 
Projecting  $\ol M$ from the point $[0:0:1:0]\notin \ol M$ onto 
$\PP_{[x:y:w]}(1,1,1)$ is an isomorphism under which the divisor 
at infinity is the line $(w=0).$

The compactifications described in Proposition \ref{compact2} are smooth 
if and only if $a\frac{n}{k}=b\frac{n}{k}=1.$ This is possible only when 
$a=b=\frac nk=1.$ Moreover, only the case $n>1$ is relevant here, 
otherwise the  compactifications in Propositions \ref{compact2} and 
\ref{compactific} coincide.

Using again the homogeneity condition (\ref{hom}), we see that $dnc=2.$ 
Therefore we obtain $n=2,$ which yields $k=2,$ and  $d=c=1.$ In this case, 
we compactify the deformation of a $\frac14(1,1)-$singularity, and the 
compactification is 
$$
\ol M=\ol N/\mu_2\subset \PP(1,1,1,1)/\mu_2,
$$
where $\ol N=(xy=w^2P(\frac{z^2}{w^2})).$ In this case $\ol N$ is a smooth 
quadric, and hence isomorphic to $\PP^1\times \PP^1,$ and the curve at 
infinity is the diagonal.The quotient map $\ol N\to \ol N/\mu_2\simeq\ol M$ 
is ramified along the diagonal. Therefore $\ol M$ is the projective plane $\PP^2$ 
and the curve at infinity is a smooth conic.

The compactifications described in Proposition \ref{comp-def-non-cyclic} 
are never smooth.
\end{proof}

\subsection*{Acknowledgements} While writing this paper, the second author 
was partially supported by the NSF grant DMS-1007114 .

\bibliographystyle{alpha} 

\end{document}